\documentclass[11pt,a4paper,reqno]{amsart}
\usepackage[latin1]{inputenc}
\usepackage[english]{babel}
\usepackage{amsmath}
\usepackage{amsfonts}
\usepackage{amssymb}
\usepackage{mathrsfs}
\usepackage{latexsym}
\usepackage{yfonts}
\usepackage{natbib}

\usepackage[margin=2.5cm]{geometry}

\usepackage{color}

\usepackage{graphicx,color}

\usepackage[plainpages=false,colorlinks,hyperindex,bookmarksopen,linkcolor=red,citecolor=blue,urlcolor=blue]{hyperref}

\bibpunct{[}{]}{;}{n}{,}{,}

\newtheorem{tm}{Theorem}[section]

\newtheorem{defin}{Definition}

\newtheorem{prop}{Proposition}[section]
\newtheorem{remark}{Remark}[section]

\numberwithin{equation}{section}

\allowdisplaybreaks

\title{Fractional Cauchy problem on random snowflakes}
	
	\author{Raffaela Capitanelli* }
	\address{Department of Basic and Applied Sciences for Engineering\newline Sapienza University of Rome\newline via A. Scarpa 10, Rome, Italy}
	\email{raffaela.capitanelli@uniroma1.it}

	\author{Mirko D'Ovidio}
	\address{Department of Basic and Applied Sciences for Engineering\newline Sapienza University of Rome\newline via A. Scarpa 10, Rome, Italy}
	\email{mirko.dovidio@uniroma1.it}
\bigskip

\email[Corresponding author]{raffaela.capitanelli@uniroma1.it}

\begin{document}

\date{\today}
\maketitle

\begin{abstract}
We consider time-changed Brownian motions on random Koch (pre-fractal and fractal) domains where the time change is given by the inverse to a subordinator. In particular, we study the fractional Cauchy problem  with Robin condition on the pre-fractal boundary obtaining asymptotic results for the corresponding fractional diffusions with Robin, Neumann and Dirichlet boundary conditions on the fractal domain. 
\end{abstract}


{\bf AMS-MSC:} 26A33, 35R11, 60J50, 58J37 .

\section{Introduction}

Many physical and biological phenomena take place across irregular  and wild structures  in which boundaries are \lq\lq large\rq\rq while bulk is \lq\lq small\rq\rq. In this framework, domains  with fractal  boundaries provide a suitable setting to model phenomena in which the surface effects are enhanced like, for example,  pulmonary system, root  infiltration, tree foliage, etc.. 

  In this paper, we consider random Koch domains which are domains whose boundary are constructed by mixtures of Koch curves with random scales. These domains are obtained as limit of domains with Lipschitz boundary whereas for the limit object, the fractal given by the random Koch domain, the boundary has Hausdorff dimension between $1$ and $2$.

Our attention will be focused on fractional Cauchy problems on the random Koch domains with boundary conditions.

Literature on fractional Cauchy problems is extensive both from the probability and the analysis point of view. Here, our aim  is not  providing a large list of references. We mention here only few works investigating basic and fondamental aspects: \cite{ACV},     \cite{BAZ}, \cite{CMN}, \cite{DovNane}, \cite{EK}, \cite{KLW}, \cite{MNV}, \cite{OB},  \cite{Zac}.

The non-local time-operator we deal with is very general and covers a huge class of non-local (convolution type) operators. Such operators have been recently considered in the papers \cite{chen, toaldo}. From the probabilistic point of view, we consider time-changed Brownian motions  where the time change is given by an inverse to a subordinator characterized by a symbol which is a Bernstein function. Thus, with this time-fractional operator at hand, we study the fractional Cauchy problem with Robin condition on the pre-fractal boundary and we obtain asymptotic results for the corresponding fractional diffusions with Robin, Neumann and Dirichlet boundary conditions on the fractal domain.

The asymptotic problem we deal with can be illustrated, in the simple case, by the following parabolic Dirichlet-Robin problem on the interval $(0,a)$, $a>0$. More precisely, we consider the heat equation 
\begin{align}
& \partial_t u_n = \partial_{xx} u_n, \quad t>0, \, x \in (0,a)\notag \\
& u_n(t, 0) = 0, \quad t>0 \label{pEsempio}\\
& u_n(0, x) =  f(x), \quad x \in (0,a)\notag 
\end{align}
with Robin boundary condition
\begin{align}
& \partial_x u_n(t, a) + c_n u_n(t, a)=0, \quad t>0
\end{align}
where $c_n >0$, $n \in \mathbb{N}$. The solution can be written as follows $u_n(t,x) = \sum_{k \geq 1} e^{- t \lambda_k^{(n)}} \phi_k^{(n)}(x)\, f_k$ where $f_k^{(n)} = \int f(x) \phi_k^{(n)}(x)\, dx$, $k \in \mathbb{N}$. Notice that $\phi_k^{(n)}(x)=\sin(x\sqrt{\lambda_k^{(n)}})$ and $\sqrt{\lambda_k^{(n)}} = z_k^{(n)}$ are the eigenvalues associated to $\phi_k^{(n)}$ where $z_k^{(n)}$ are solutions to $\tan (a z_k^{(n)}) = - z_k^{(n)}/ c_n$, $k \in \mathbb{N}$.

Our aim here is to point out the asymptotic behaviour of the solution $u_n$ as $n \to \infty$. We obtain three different limit problems. If $c_n \to 0$, then $z_k^{(n)} \to z_k^N= \left( \frac{\pi}{2}+ \pi k\right) \frac{1}{a}$ and therefore $u_n(t,x) \to u(t,x) = \sum_{k \geq 1} e^{- t (z_k^N)^2} \sin(x z_k^N)\, f_k$ where $f_k = \int f(x) \sin(x z_k^N) dx
$ and $u$ is the solution to \eqref{pEsempio} with Neumann condition
\begin{align*}
\partial_x u(t, a) =0, \quad t>0.
\end{align*}
If $c_n \to \infty$, 
then $z_k^{(n)}  \to z_k^D = \frac{\pi k}{a}$  and therefore  $u_n \to u$ where the solution $u(t,x) = \sum_{k \geq 1} e^{- t (z_k^D)^2} \sin(x z_k^D)\, f_k $ with $f_k = \int f(x) \sin(x z_k^D) dx$ solves \eqref{pEsempio} with Dirichlet condition
\begin{align*}
u(t, a) =0, \quad t>0.
\end{align*}
If $c_n \to c \in (0,\infty)$, then  
$z_k^{(n)} \to z_k^R >0\, :\, \tan (a z_k^R) = - \frac{z_k^R}{c}
$ and therefore $u_n \to u$ where the solution $u(t,x) = \sum_{k \geq 1} e^{- t (z_k^R)^2} \sin(x z_k^R)\, f_k$ with $f_k = \int f(x) \sin(x z_k^R) dx$ solves \eqref{pEsempio} with Robin boundary condition
\begin{align*}
\partial_x u(t, a) + c u(t,a) =0, \quad t>0.
\end{align*}
Now we wonder if a similar asymptotic behaviour holds for the analogue time-fractional problem. A simple example is given by the problem
\begin{align*}
& \partial^\beta_t u_n = \partial_{xx} u_n, \quad t>0, \, x \in (0,a)\\
& \partial_x u_n(t, a) + c_n u_n(t, a)=0, \quad t>0\\
& u_n(t, 0) = 0, \quad t>0\\
& u_n(0, x) =  f(x), \quad x \in (0,a)
\end{align*}
with $c_n >0$, $n \in \mathbb{N}$ where $\partial_t^\beta u$ is the Caputo fractional derivative of $u$ (see formula \eqref{der:Caputo} below). The solution can be written as follows
\begin{align}
\label{unfract}
u_n(t,x) = \sum_{k \geq 1} E_\beta (- \lambda_k^{(n)} t^\beta) \phi_k^{(n)}(x)\, f_k^{(n)}
\end{align}
where 
\begin{align*}
E_{\beta}(w) = \sum_{k \geq 0} \frac{w^\beta}{\Gamma(\beta k +1)}, \quad w \geq 0
\end{align*}
is the Mittag-Leffler function and the system $\{\phi_k^{(n)}, \lambda_k^{(n)}: k \in \mathbb{N}\}$ has been introduced before. By simple arguments, we get that the solution \eqref{unfract} uniformly converges to a function $u$ which turns out to be analogously related to the boundary problems above (Neumann, Dirichlet, Robin) with the Caputo time-fractional derivative $\partial_t^\beta$ in place of the ordinary derivative $\partial_t$. This is due to the fact that we have explicit representation of the system $\{\phi_k^{(n)}, \lambda_k^{(n)}: k \in \mathbb{N}\}$. 

Following the same spirit, in the present paper, we move on to general domains like the random snowflakes we have introduced before and we address the same asymptotic problem with a general time-fractional operator. In this case we do not have the same informations about the associated system and the compact representation of the solution. We overcome this difficulty by using the theory of Dirichlet forms and Markov processes. An essential tool will be given by the convergence of forms associated with time-changed processes.

We remark that the peculiarity in studying the asymptotic behaviour of these approximating problems is that one has to deal with an increasing sequence of Lipschitzian domains  which converges in the limit to the domain whose boundary is  a fractal.

The plan of the paper is the following: in Section 2 we introduce the random Koch domains; in Section 3 we recall the definition of Dirichlet forms with associated base processes; in Section 4 we introduce time-fractional equations and time changes; in the last section we prove our main results. More precisely, in Theorem 5.1 we solve the asymptotic problem for the the time-changed processes and in Theorem 5.2 we point out some peculiar aspects arising by passing from the ordinary to the fractional Cauchy problem.

\section{Random Koch domains (RKD)}
\label{section:RKD}
We first introduce the Koch (snowflake) domain and then we construct the random Koch domains. Let $\ell_{a} \in (2,4)$ with $a \in I \subset \mathbb{N}$ be the reciprocal of the contraction factor for the family $\Psi^{(a)}$ of contractive similitudes $\psi^{(a)}_i : \mathbb{C} \to \mathbb{C}$ given by
\begin{align*}
\psi^{(a)}_1(z)=\frac{z}{\ell_a}, \quad \psi^{(a)}_2(z)= \frac{z}{\ell_a} e^{\imath \theta(\ell_a)} + \frac{1}{\ell_a},
\end{align*}
\begin{align*}
\psi^{(a)}_3(z) = \frac{z}{\ell_a} e^{\imath \theta(\ell_a)} + \frac{1}{2} + \imath \sqrt{\frac{1}{\ell_a} - \frac{1}{4}}, \quad \psi^{(a)}_4(z) = \frac{z-1}{\ell_a} + 1
\end{align*}
where $\theta(\ell_a) = \arcsin(\sqrt{\ell_a(4-\ell_a)}/2)$. Let $\Xi =I^\mathbb{N}$ with $I \subset \mathbb{N}$, $|I|=N$, and let $\xi=(\xi_1, \xi_2, \ldots) \in \Xi$.  We call $\xi$ an environment sequence where $\xi_n$ says which family of contractive similitudes we are using at level $n$. Set $\ell^{(\xi)}(0)=1$ and 
\begin{align}
\ell^{(\xi)}(n)= \prod_{i=1}^n \ell_{\xi_i}.
\end{align}
We define a left shift $S$ on $\Xi$ such that if $\xi = \left( \xi_1, \xi_2,\xi_3, \ldots \right),$ then $S\xi = \left(  \xi_2,\xi_3, \ldots \right).$ For $B\subset \mathbb{R}^2$ set 
$$\Upsilon^{(a)}(B) = \bigcup_{i =1}^{4} \psi^{( a)}_i \left( B \right)$$
and
$$\Upsilon^{(\xi)}_n(B) =\Upsilon^{(\xi_1)}\circ\dots\circ\Upsilon^{(\xi_n)} \left( B \right).$$
The fractal $K^{( \xi )}$ associated with the environment sequence $\xi$ is defined by
$$K^{( \xi )}=\overline{\bigcup_{n = 1}^{+ \infty}\Upsilon_n^{(\xi)}(\Gamma)}$$
where $\Gamma=\{ P_1, P_2\}$ with $P_1=(0,0)$ and $P_2=(1,0).$ We remark that these fractals do not have any exact self-similarity, that is, there is no scaling factor which leaves the set invariant:
however, the family $\{K^{( \xi )}, \xi\in\Xi\}$ satisfies the following relation 
\begin{equation}\label{ss}
K^{( \xi )}=\Upsilon^{(\xi_1)}(K^{( S\xi )}).
\end{equation} 
Moreover, the spatial symmetry is preserved and  the set $K^{( \xi )}$ is locally spatially homogeneous, that is,  the volume measure $\mu^{(\xi)} $ on $K^{( \xi )}$ satisfies the locally spatially homogeneous condition (\ref{ssm}) below. Before  describing this measure, we introduce some notations.
For $\xi\in \Xi,$ we define the word space $$W=W^{(\xi)}=\{ (w_1,w_2,...): 1\leqslant w_i\leqslant 4\}$$ and, for $w\in W,$ we set $w|n=(w_1,...,w_n)$ and $\psi_{w | n}=\psi^{(\xi_1)}_{w_1}\circ\dots\circ\psi^{(\xi_n)}_{w_n}.$ The volume measure  $\mu^{(\xi)} $ is  the unique Radon measure on $K^{(\xi)}$ such that \begin{equation}\label{ssm}\mu^{(\xi)}(\psi_{w | n}(K^{(S^n\xi)}))=\frac{1}{4^n}\end{equation} for all $w\in W,$
(see Section 2 in \cite{BH}) as,  for each $a\in A,$ the family $\Psi^{(a)} $ has $4$  contractive similitudes. Let $K_0$ be the line segment of unit length with $P_1=(0,0)$ and $P_2=(1,0)$  as endpoints. We set, for each $n \in \mathbb{N}$, 
$$K^{(\xi)}_n=\Upsilon_n^{(\xi)}(K_0)$$
and ${K^{(\xi)}_n}$ is the so-called ${n}$\text{-th} \text{prefractal curve}.

Let us consider the random vector $\boldsymbol{\xi}= (\boldsymbol{\xi}_1, \boldsymbol{\xi}_2, \ldots)$ whose components $\boldsymbol{\xi}_i$ take values on $I$ with probability mass function ${\bf P} : \Xi \to [0,1]$. Thus, the construction of the random $n$-th pre-fractal  curve 
$$K^{(\boldsymbol{\xi})}_n=\Upsilon_n^{(\boldsymbol{\xi})}(K_0) $$
depends on the realization of $\boldsymbol{\xi}$ with probability $P(\boldsymbol{\xi}_i=\xi_i)$ for its $i$-th  component. We assume that $\{\boldsymbol{\xi}_i\}_{i=1, \ldots, n}$ are identically distributed and $\boldsymbol{\xi}_i \perp \boldsymbol{\xi}_j$ for $i\neq j$, that is we obtain the curve $K_n^{(\xi)}$ with probability
\begin{align*}
{\bf P}(\boldsymbol{\xi}| n =\xi | n) = \prod_{i=1}^n {\bf P}(\boldsymbol{\xi}_i=\xi_i)
\end{align*}
where $\boldsymbol{\xi}| n = (\boldsymbol{\xi}_1, \ldots, \boldsymbol{\xi}_n)$ and $\xi |n = (\xi_1, \ldots, \xi_n)$. Further on we only use the superscript $(\xi |n)$ or $(\boldsymbol{\xi}|n)$ in order to streamline the notation.

The fractal $K^{(\boldsymbol{\xi})}$ associated with the random environment sequence $\boldsymbol{\xi}$ is therefore defined by
$$K^{(\boldsymbol{\xi})}=\overline{\bigcup_{n = 1}^{+ \infty}\Upsilon_n^{(\boldsymbol{\xi})}(\Gamma)}$$
where $\Gamma=\{ P_1, P_2\}$ with $P_1=(0,0)$ and $P_2=(1,0).$

Let $\Omega^{(\xi |n)}$ be the planar domain obtained from a regular polygon by replacing each side with a pre-fractal curve $K_n^{(\xi)}$ and $\Omega^{(\xi)}$ be the planar domain obtained by replacing each side with the corresponding fractal curve $K^{(\xi)}$. We introduce the random planar domains $\Omega^{(\boldsymbol{\xi}|n)}$ and $\Omega^{(\boldsymbol{\xi})}$ by considering the random curves $K_n^{(\boldsymbol{\xi})}$ and $K^{(\boldsymbol{\xi})}$. Examples of (pre-fractal) random Koch domains are given in figures \ref{fig-outside} (outward curves), \ref{fig-con} (inward curves), \ref{fig-con-2} (inward curves) by choosing as regular polygon the square. 

\begin{figure}
\caption{Outward curves}
\centering
\includegraphics[scale=.5]{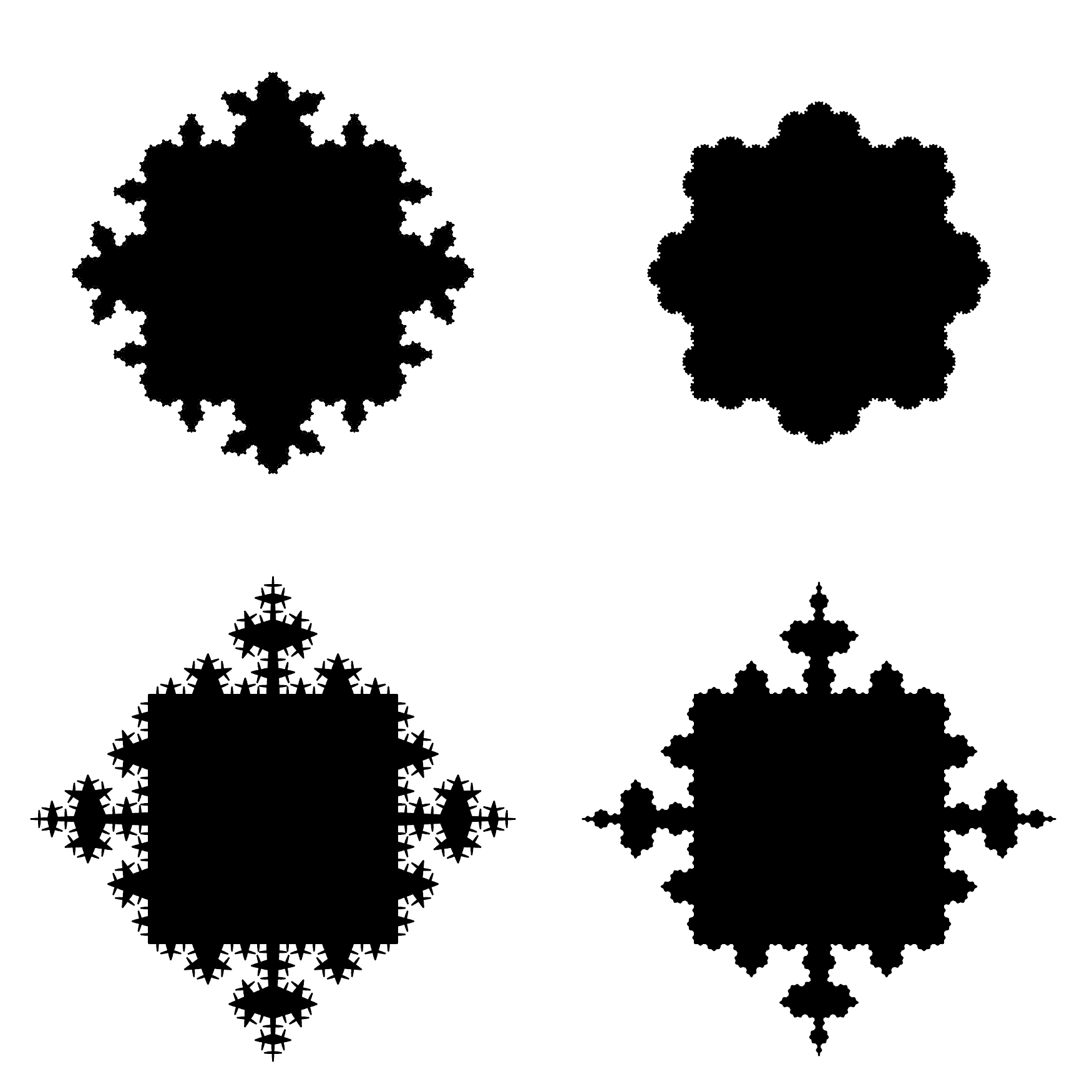}
\label{fig-outside}
\end{figure}

\begin{figure}
\caption{Inward curves}
\centering
\includegraphics[scale=.6]{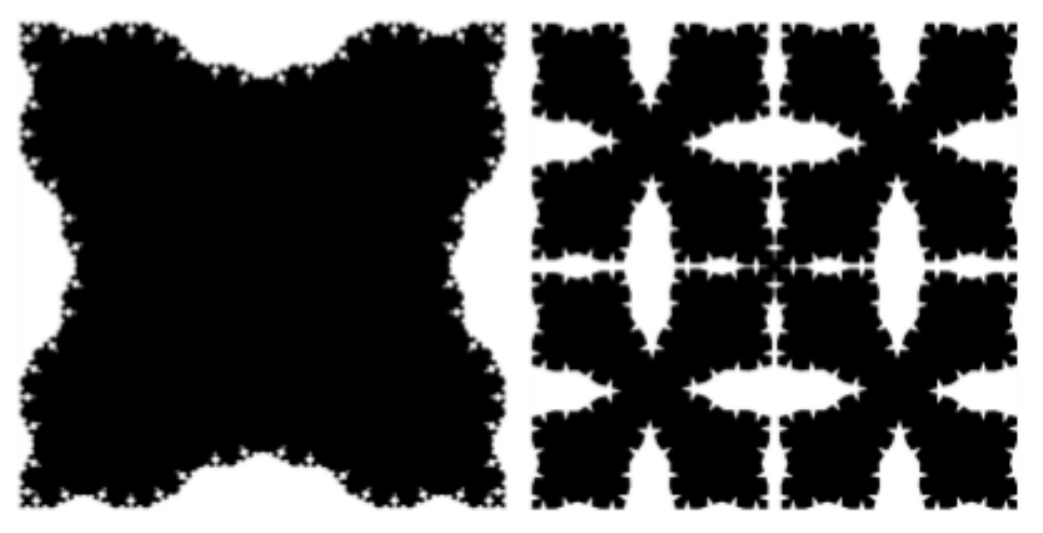}
\label{fig-con}
\end{figure}

\begin{figure}
\caption{Inward curves}
\centering
\includegraphics[scale=.6]{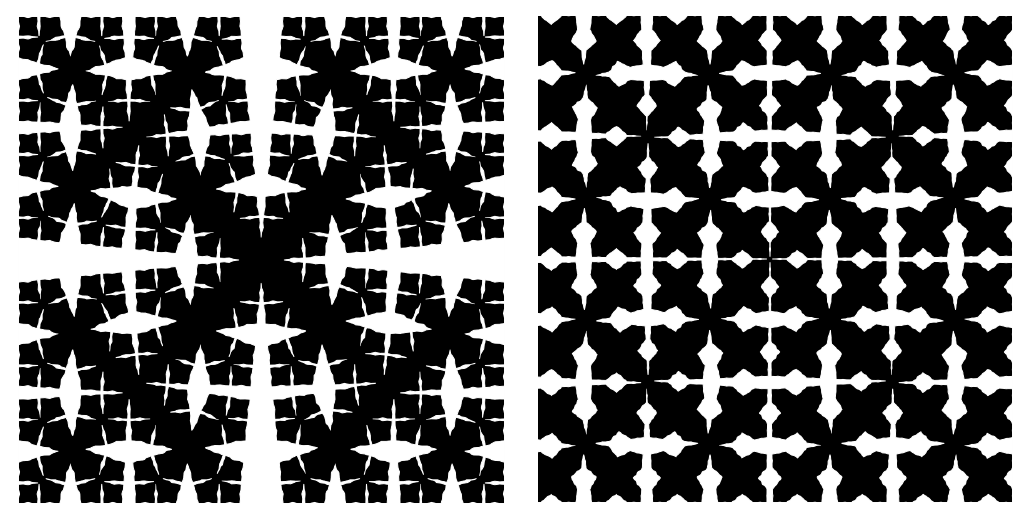}
\label{fig-con-2}
\end{figure}

%

Since $\boldsymbol{\xi}_i \stackrel{law}{=} \boldsymbol{\xi}_1$, $\forall\, i$, we have that the Hausdorff dimension $d^{(\boldsymbol{\xi})}$ of the curve  $K^{(\boldsymbol{\xi})}$ can be obtained by considering the strong law of large numbers and the fact that
\begin{align*}
\frac{\ln 4^n}{\sum_{i=1}^n \ell_{\boldsymbol{\xi}_i}}  = \frac{\ln 4}{\frac{1}{n} \sum_{i=1}^n \ell_{\boldsymbol{\xi}_i}} \stackrel{a.s.}{\to} \frac{\ln 4}{\mathbf{E}[\ln \ell_{\boldsymbol{\xi}_1}]}, \quad n \to \infty.
\end{align*}
Then (see \cite[Lemma 2.3]{BH}),
\begin{align}
d^{(\boldsymbol{\xi})} = \frac{\ln 4}{\ln \prod_{a \in I} (\ell_{a})^{P(\boldsymbol{\xi}_1 =a)}} = \frac{\ln 4}{\mathbf{E}[ \ln \ell_{\boldsymbol{\xi}_1}]}. \label{dimHK}
\end{align}

Moreover the measure $\mu^{(\xi)}$ in \eqref{ssm} has the property that there exist two positive constants $C_1, C_2,$ such that,
\begin{equation}
\label{eq:6} 
C_1 r^{d^{(\xi)} }\le \mu^{{(\xi)} }(\mathcal{B}(P,r)\cap K^{(\xi)})\le C_2 r^{d^{(\xi)} }\ ,\quad
\forall\,P\in K^{(\xi)}, 
\end{equation} 
where $\mathcal{B}(P,r)$ denotes the Euclidean ball with center in $P$ and radius $0<r\leq1$ (see  \cite{BH}).  According to Jonsson and Wallin (see  \cite{JonWal}), we say that $K^{(\xi)} $ is a $d$-set with respect to  the Hausdorff measure $\mathcal{H}^d,$ with $d=d^{(\xi)}.$ The sequence
\begin{align}
\sigma^{(\xi | n)} = \frac{\ell^{(\xi | n)}}{4^n}, \qquad \textrm{where} \qquad \ell^{(\xi | n)} = \prod_{i=1}^n \ell_{\xi_i} \label{seq-sig}
\end{align}
is obtained from the realization of $\boldsymbol{\xi}|n$ and therefore, from the realization of the random variable $\ell^{(\boldsymbol{\xi}|n)}$ with mean value given by
\begin{align*}
\mathbf{E}[\ell^{(\boldsymbol{\xi} | n)}] = \prod_{i=1}^n \mathbf{E} [\ell_{\boldsymbol{\xi}_i}] = \left( \mathbf{E}[ \ell_{\boldsymbol{\xi}_1}] \right)^n.
\end{align*}
Thus, for $\alpha=\mathbf{E} \ell_{\boldsymbol{\xi}_1} \in (2,4)$ we find the mean value $\mathbf{E}[\sigma^{(\boldsymbol{\xi} | n)}] = \alpha^n / 4^n$.

The realization $\xi|n$ can be regarded as the vector $\mathbf{a}|n=(a_1, \ldots, a_n)$ which is a $n$-dimensional vector with $N$ different values of $I$, that is $\mathbf{a}|n \in I^n$. We introduce the multinomial distribution 
\begin{align*}
p_{\mathbf{a}|n} = n! \prod_{i=1}^N \frac{p_i^{\sharp(a_i)}}{\sharp(a_i) !}, \qquad \sum_{i=1}^N \sharp(a_i) =n, \qquad \sum_{i=1}^N p_i =1
\end{align*}
where $p_i=P(\boldsymbol{\xi}_1= a_i)$ and write $p=\{p_i\}_{i=1, \ldots, N}$. Thus, for the realization of the vector $\boldsymbol{\xi}|n$ we have that
\begin{align*}
\mathbf{E} [ \mathbf{1}_{\overline{\Omega^{(\boldsymbol{\xi}|n)}}}] = \sum_{\xi |n } p_{\xi |n} \mathbf{1}_{\overline{\Omega^{(\xi |n)}}}
\quad \textrm{with} \quad p_{\xi |n} = P(\boldsymbol{\xi}|n = \xi |n).
\end{align*}
or equivalently
\begin{align*}
\mathbf{E} [ \mathbf{1}_{\overline{\Omega^{(\boldsymbol{\xi}|n)}}}] = \sum_{\mathbf{a} | n} p_{\mathbf{a}|n} \mathbf{1}_{\overline{\Omega^{(\mathbf{a}|n)}}} .
\end{align*}
We notice that
\begin{align}
\mathbf{E} [ \mathbf{1}_{\overline{\Omega^{(\boldsymbol{\xi}|n)}}}(x) ] =  1 = (p_1+\ldots + p_N)^N \quad \textrm{if} \quad  x \in \bigcap_{\Xi} \bigcup_{i=1}^n \overline{\Omega^{(\xi |i)}}. \label{indic-random-set}
\end{align}

\section{Dirichlet forms and Base processes}

Let $E$ be a locally compact, separable metric space   and $E_\partial= E \cup \{\partial\}$ be the one-point compactification of $E$. Denote by $\mathcal{B}(E)$ the $\sigma$-field of the Borel sets in $E$ ($\mathcal{B}_\partial$ is the $\sigma$-field in $E_\partial$).  Let $X=\{X_t, t\geq 0 \}$ with infinitesimal generator $(A, D(A))$ be the symmetric Markov process on $(E, \mathcal{B}(E))$ with transition function $p(t, x, B)$ on $[0, \infty) \times E \times \mathcal{B}(E)$. The point $\partial$ is the cemetery point for $X$ and a function $f$ on $E$ can be extended to $E_\partial$ by setting $f(\partial)=0$. The associated semigroup is uniquely defined by
\begin{align*}
P_t f(x):= \int_E f(y) p(t,x,dy) = \mathbf{E}_x[f(X_t)], \quad f \in C_\infty(E)
\end{align*}
with $X_0=x \in E$ where $\mathbf{E}_x$ denote the mean value with respect to the probability measure
\begin{align*}
\mathbf{P}_x(X_t \in dy) = p(t,x,dy) 
\end{align*}
and $C_\infty$ is the set of continuous function $C(E)$ on $E$ such that $f(x)\to 0$ as $x\to \partial$. Let $\mathcal{E}(u,v)= (\sqrt{-A}u, \sqrt{-A}v)$ with domain $D(\mathcal{E})= D(\sqrt{-A})$ be the Dirichlet form associated with (the non-positive definite, self-adjoint operator) $A$. Then $X$ is equivalent to an $m$-symmetric Hunt process whose Dirichlet form $(\mathcal{E}, D(\mathcal{E}))$ is on $L^2(E)$ (see  the books \cite{chen-book, FUK-book}). Without restrictions we assume that the form is regular (\cite[page 143]{FUK-book}).\\

We say that $X$ is the base process. Our aim is to consider time changes of the base process $X$. Such random times will be introduced in the next section.

\section{Time fractional equations and Time changes}

We first introduce the subordinator $H=\{H_t, t\geq 0 \}$ for which
$$\mathbf{E}_0[\exp( - \lambda H_t)] = \exp(- t \Phi(\lambda))$$ 
where $\Phi$ is the symbol of $H$. The symbol $\Phi$ may be associated also to the inverse $L$ of $H$, that is $L=\{L_t , t\geq 0\}$ defined as
\begin{align*}
L_t = \inf \{ s \geq 0\, :\, H_s > t \}, \quad t\geq 0.
\end{align*}
We assume that $H_0=0$, $L_0=0$. By definition, we also have that
\begin{align}
\label{relationHL}
\mathbf{P}_0(H_t < s) = \mathbf{P}_0(L_s>t), \quad s,t>0.
\end{align}
The symbol $\Phi$ we consider hereafter is a Bernstein function with representation
\begin{equation}
\Phi(\lambda) = \int_0^\infty \left( 1 - e^{ - \lambda z} \right) \Pi(dz), \quad \lambda \geq 0 \label{LevKinFormula}
\end{equation} 
where $\Pi$ on $(0, \infty)$ with $\int_0^\infty (1 \wedge z) \Pi(dz) < \infty$ is the associated \emph{L\'evy measure}.  We also recall that
\begin{align}
\label{tailSymb}
\frac{\Phi(\lambda)}{\lambda} = \int_0^\infty e^{-\lambda z} \overline{\Pi}(z)dz, \qquad \overline{\Pi}(z) = \Pi((z, \infty))
\end{align}
and $\overline{\Pi}$ is the so called \emph{tail of the L\'evy measure}. Both random times $H,L$ are non-decreasing.  We do not consider step-processes with $\Pi((0, \infty)) < \infty$ and therefore we focus only on strictly increasing subordinators with infinite measures. Thus, the inverse process $L$ turns out to be a continuous process. For details, see the books \cite{bertoin, BerBook}.\\

We now introduce the fractional operators and the fractional equations governing the time-changed process $X^L = \{ X \circ L_t, t\geq 0\}$, that is the base process $X=\{ X_t, t\geq 0\}$ with the time change $L$ characterized by the symbol $\Phi$.\\

Let $M>0$ and $w\geq 0$. Let $\mathcal{M}_w$ be the set of (piecewise) continuous function on $[0, \infty)$ of exponential order $w$ such that $|u(t)| \leq M e^{wt}$. Denote by $\widetilde{u}$ the Laplace transform of $u$. Then, we define the operator $\mathfrak{D}^\Phi_t : \mathcal{M}_w \mapsto \mathcal{M}_w$ such that
\begin{align*}
\int_0^\infty e^{-\lambda t} \mathfrak{D}^\Phi_t u(t)\, dt = \Phi(\lambda) \widetilde{u}(\lambda) - \frac{\Phi(\lambda)}{\lambda} u(0), \quad \lambda > w
\end{align*}
where $\Phi$ is given in \eqref{LevKinFormula}. Since $u$ is exponentially bounded, the integral $\widetilde{u}$ is absolutely convergent for $\lambda>w$.  By Lerch's theorem the inverse Laplace transforms $u$ and $\mathfrak{D}^\Phi_tu$ are uniquely defined. Notice that
\begin{align}
\label{PhiConv}
\Phi(\lambda) \widetilde{u}(\lambda) - \frac{\Phi(\lambda)}{\lambda} u(0) = & \left( \lambda \widetilde{u}(\lambda) - u(0) \right) \frac{\Phi(\lambda)}{\lambda}.
\end{align}
Simple arguments say that $\mathfrak{D}^\Phi_t$ can be written as a convolution involving the ordinary derivative and the inverse transform of \eqref{tailSymb} iff $u \in \mathcal{M}_w \cap C([0, \infty), \mathbb{R}_+)$ and $u^\prime \in \mathcal{M}_w$, that is,
\begin{align*}
\mathfrak{D}^\Phi_t u(t) = \int_0^t u^\prime(s) \overline{\Pi}(t-s)ds.
\end{align*}
We notice that when $\Phi(\lambda)=\lambda$ (that is, the ordinary derivative) we have that a.s. $H_t = t$ and $L_t=t$. 

We also notice that for $\Phi(\lambda)=\lambda^\beta$, the symbol of a stable subordinator, the operator $\mathfrak{D}^\Phi_t$ becomes the Caputo fractional derivative
\begin{align}
\label{der:Caputo}
\mathfrak{D}^\Phi_t u(t) = \frac{1}{\Gamma(1-\beta)} \int_0^t \frac{u^\prime(s)}{(t-s)^\beta}ds
\end{align}
with $u^\prime(s)=du/ds$.

For $\Phi(\lambda) = (\lambda +\eta)^\beta - \eta^\beta$, with $\eta\geq 0$ and $\beta\in (0,1)$, the operator $\mathfrak{D}^\Phi_t$ becomes the Caputo tempered fractional derivative
\begin{align*}
\mathfrak{D}^\Phi_t u(t) = 
\frac{1}{\Gamma(1-\beta)} \int_0^{t} \frac{u^\prime(s)}{(t-s)^{\beta}} e^{- \eta(t-s)} \, ds
\end{align*}
with $u^\prime(s)=du/ds$.

For explicit representation of the operator $\mathfrak{D}^\Phi_t$ see also the recent works \cite{chen, toaldo}. \\

Let $X$ be the process with generator $(A, D(A))$ introduced above. In the present work  we consider the time fractional equation
\begin{align}
\label{time-frac-problem}
\mathfrak{D}^\Phi_t u = A u, \qquad u_0=f \in D(A).
\end{align}
The probabilistic representation of the solution to \eqref{time-frac-problem} is written in terms of the time-changed process $X^{L}$, that is 
\begin{align}
\label{sol-time-frac-problem}
u(t,x) = \mathbf{E}_x[f(X^{L}_t)]= \int_0^\infty P_s f(x)\, \mathbf{P}_0(L_t \in ds), \quad x \in E, \, t>0.
\end{align}
We notice that $ \eqref{sol-time-frac-problem}$ is not a semigroup, indeed the random time $L$ is not Markovian and therefore, the composition $X^L$ is not a Markov process. 

The fractional Cauchy problem has been investigated by many authors by considering Caputo derivative and only recently, by taking into account more general operators. The following theorem has been obtained in \cite{CapDovVIA} for Feller processes (not necessarily Feller diffusions, see \cite{CapDovVIA}) and we mention here such a result for the reader's convenience.
\begin{tm}
\label{time-frac-THM}
The function \eqref{sol-time-frac-problem} is the unique strong solution in $L^2(E)$ to \eqref{time-frac-problem} in the sense that:
\begin{enumerate}
\item $\varphi: t \mapsto u(t, \cdot)$ is such that $\varphi \in C([0, \infty), \mathbb{R}_+)$ and $\varphi^\prime  \in \mathcal{M}_0$,
\item $\vartheta : x \mapsto u(\cdot, x)$ is such that $\vartheta, A\vartheta \in D(A)$,
\item $\forall\, t > 0$, $\mathfrak{D}^\Phi_t u(t,x) = Au(t,x)$ holds a.e in $E$
\item $\forall\, x \in E$, $u(t,x) \to f(x)$ as $t \downarrow 0$.
\end{enumerate}
\end{tm}

In \cite{chen} the author proves existence and uniqueness of strong solutions to general time fractional equations with initial datum $f \in D(A)$. In \cite{CKKW} the authors establish existence and uniqueness for weak solutions and initial datum $f \in L^2$. The result in Theorem \ref{time-frac-THM} has been proved in a general setting, that is by considering a generator of a Feller process as in \cite{chen} but following a very different approach. We notice that the condition on the initial datum $f$ must be better specified  for the compact representation of the solution, this is the case investigated in \cite{DovNane} for instance (the domain has no boundary) or the case investigated in \cite{CMN} (with Dirichlet condition on the boundary).\\

In the next section we will study continuous base processes time changed by continuous random times, thus we do not stress the fact that the previous result holds for Feller process (right-continuous with no discontinuity other than jumps). 




\section{Main results}
We consider the prefractal RKD $\Omega^{(\xi | n)}$ defined in Section \ref{section:RKD} and we construct the set $\Omega^{(\xi|n)} \setminus B$ where $B \subset \Omega^{(\xi | 1)}$ is a ball. 

 Then, we consider Brownian diffusions on the Random Koch Domain $\Omega^{(\xi|n)} \setminus B$. Let $X^n=\{X_t^n, t\geq 0\}$ with $X_0^n=x \in \Omega^{(\xi | n)} \setminus B$ be a sequence of planar Brownian motions for a given $\xi \in \Xi$. Let $(A^n, D(A^n))$ be the generator of $X^n$, in particular $A^n = \Delta$ and
 $$D(A^n) =$$
\begin{align*}
 \{u \in H^1(\Omega^{(\xi | n)}\setminus B): \, \Delta u \in L^2(\Omega^{(\xi | n)}\setminus B),\, u|_{\partial B}=0,\, (\partial_{\bf n} u + c_n \sigma^{(\xi | n)} u)|_{\partial \Omega^{(\xi | n)}}=0\}
\end{align*}
where $c_n \geq 0$, $\mathbf{n}(x)$ denote the inward normal vector at $x \in \partial \Omega^{(\xi | n)}$ and  $\sigma^{(\xi | n)} $ is defined in \eqref{seq-sig}.
 It is well-known that there is one to one correspondence between  the infinitesimal generator of $X^n$ and the closed symmetric form $(\mathcal{E}^n, D(\mathcal{E}^n)$ (see \cite[Theorem 1.3.1]{FUK-book}).

 We recall that  a form $(\mathcal{E}^n, D(\mathcal{E}^n)$ can be defined in the whole of $L^2(F, m)$ by setting 
$\mathcal{E}^n(u,u) = +\infty \quad \forall\, u \in L^2(F, m) \setminus D(\mathcal{E}^n).$ Similarly  a forms $\mathcal{E}$, $\mathcal{E}$ can be defined in the whole of $L^2(F, m)$ by setting 
$\mathcal{E}(u,u) = +\infty \quad \forall\, u \in L^2(F, m) \setminus D(\mathcal{E}).$

For the convenience of the readers we recall the definition of convergence of forms introduced by Mosco in \cite{Mosc}, denoted by $M$-convergence.

\begin{defin}
\label{d2.4} A sequence of forms $\{\mathcal{E}^n(\cdot,\cdot)\}$ $M$-converges to a form $\mathcal{E}(\cdot,\cdot)$ in $L^2(F)$ if
\begin{itemize}
\item[{\bf (a)}] For every $v_n$ converging weakly to $u$ in $L^2(F)$
\begin{equation}
\label{e2.7} \underline{\lim}\ \mathcal{E}^n(v_n,v_n)\ge \mathcal{E}(u,u)\ ,\quad
\hbox{\rm as } n\to\infty. \end{equation}
\item[{\bf (b)}] For every $u\in L^2(F)$ there exists $v_n$ converging strongly
in $L^2(F)$ such that \begin{equation}
\label{e2.8} \overline{\lim}\ \mathcal{E}^n(v_n,v_n
)\leq \mathcal{E}(u,u)\ ,\quad
\hbox{\rm as }n\to\infty\ . \end{equation}
\end{itemize}
\end{defin}

In our framework, we consider the pre-fractal form $\mathcal{E}^n(\cdot,\cdot)$ on  $L^2(\Omega^{(\xi)}\setminus B)$ by defining

$$
\mathcal{E}^n(u,u)=
$$
$$
\begin{cases}  \displaystyle\int_{\Omega^{(\xi |n)} \setminus B} |\nabla u|^2\, dx dy +c_n \sigma^{(\xi | n)} \int_{\partial \Omega^{(\xi |n)}}
|u|^2 ds  & \hbox{\text for}\,u |_{\Omega^{(\xi |n)}\setminus B}\in
 H^1(\Omega^{(\xi | n)}\setminus B), \, u|_{\partial B}=0
\\
+\infty\quad &\hbox{\text otherwise }.
\end{cases}
$$
We now introduce the time-changed process $X^{L,n} = X^n \circ L$ and we study the asymptotic behaviour of $X^{L,n}$ depending on the asymptotics for $c_n$. The process $X^{L,n}$ can be considered in order to study the corresponding time-fractional Cauchy problem on $\Omega^{(\xi|n)} \setminus B$
\begin{align*}
\mathfrak{D}^\Phi_t u = A^n u, \quad u_0=f \in D(A^n).
\end{align*}

Let $\mathbb{D}$ be the set of continuous functions from $[0, \infty)$ to $E_\partial = \Omega^{(\xi)}\cup \partial$ which are right continuous on $[0, \infty)$ with left limits on $(0, \infty)$. We denote by $\partial$ the cemetery point, that is $E^n_\partial$ is the one-point compactification of $E^n=\Omega^{(\xi | n)}$, $n \in \mathbb{N}$. Let $\mathbb{D}_0$ the set of non-decreasing continuous function from $[0, \infty) $ to $[0, \infty)$.

\begin{prop}
\label{teoK}
(Kurtz, \cite{KurtzRTC}. Random time change theorem). Suppose that $X^n$, $X$ are in $\mathbb{D}$ and $L^n$, $L$ are in $\mathbb{D}_0$. If $(X^n, L^n)$ converges to $(X, L)$ in distribution as $n\to \infty$, then $X^n \circ L^n$ converges to $X \circ L$ in distribution as $n \to \infty$.
\end{prop}
\begin{proof}
The proof follows from part b) of Theorem 1.1 and part a) of Lemma 2.3 in \cite{KurtzRTC}. Lemma 2.3 gives convergence for strictly increasing time changes. Since $H$ is strictly increasing, we use part c) of Theorem 1.1 and find results for $L$ which is non-decreasing and continuous. Then, part b) holds for the random time changes $L^n$.
\end{proof}

\begin{tm}
As $n\to \infty$,
\begin{align*}
X^{L, n} \to X^L \quad \textrm{in distribution in} \quad \mathbb{D} \quad \boldsymbol{\xi}-a.s. \quad \textrm{on} \quad \Omega^{(\boldsymbol{\xi})}.
\end{align*}
In particular, as $c_n \to c \geq 0$,
\begin{itemize}
\item[i)] if $c=0$, then $X^L$ is reflected on $\partial \Omega^{(\xi)}$, that is the process driven by
\begin{align*}
\mathfrak{D}^\Phi_t u = \Delta_N u, \quad u_0=f \in D(\Delta_N)
\end{align*}
where
\begin{align*}
D(\Delta_N) = \{u \in H^1(\Omega^{(\xi)}\setminus B): \, \Delta u \in L^2(\Omega^{(\xi)}\setminus B),\, u|_{\partial B}=0,\, (\partial_{\bf n} u)|_{\partial \Omega^{(\xi)}}=0\};
\end{align*}
\item[ii)] if $c \in (0, \infty)$, then $X^L$ is (elastic) partially reflected on $\partial \Omega^{(\xi)}$, that is the process driven by
\begin{align*}
\mathfrak{D}^\Phi_t u = \Delta_R u, \quad u_0=f \in D(\Delta_R)
\end{align*}
where
\begin{align*}
D(\Delta_R) = \{u \in H^1(\Omega^{(\xi)}\setminus B): \, \Delta u \in L^2(\Omega^{(\xi)}\setminus B),\, u|_{\partial B}=0,\, (\partial_{\bf n} u + c u)|_{\partial \Omega^{(\xi)}}=0\};
\end{align*}
\item[ii)] if $c=\infty$, then $X^L$ is killed on $\partial \Omega^{(\xi)}$, that is the process driven by
\begin{align*}
\mathfrak{D}^\Phi_t u = \Delta_D u, \quad u_0=f \in D(\Delta_D)
\end{align*}
where
\begin{align*}
D(\Delta_D) = \{u \in H^1(\Omega^{(\xi)}\setminus B): \, \Delta u \in L^2(\Omega^{(\xi)}\setminus B),\, u|_{\partial B}=0,\, u|_{\partial \Omega^{(\xi)}}=0\}.
\end{align*}
\end{itemize}
\end{tm}

\begin{remark}We point out  that the condition on the boundary $\partial \Omega^{(\xi)}$ must be meant in the dual of certain Besov spaces  (for details, see \cite{CAP0}, \cite{LV} and the references therein). 
\end{remark}

\begin{proof}
Fix $\xi \in \Xi$. First we prove the M-convergence in $L^2(\Omega^{(\xi)}\setminus B)$ of the Dirichlet forms $\mathcal{E}^n$.

The case of finite limit has been addressed  in Theorem 5.2 in  \cite{CAP}: in particular, it has been proved that 
 if  $c_n\to c\geq 0$, then the sequence of forms
$\mathcal{E}^n(\cdot,\cdot)$ $M$--converges in the space $L^2({\Omega}^{(\xi)}\setminus B)$ to the form$$
\mathcal{E}_c(u,u)=\begin{cases}  \displaystyle \int_{\Omega^{(\xi)}\setminus B} |\nabla u|^2\, dx dy +c \int_{\partial \Omega ^{(\xi)}}
|u|^2 d\mu^{(\xi)}  & \hbox{\text for}\,u|_{\Omega^{(\xi)}\setminus B}\in
 H^1(\Omega^{(\xi)}\setminus B), \, u|_{\partial B} =0\\
+\infty\quad &\hbox{\text otherwise}\end{cases}.$$

The last form $\mathcal{E}_c,$ for $c \in (0,\infty)$, is associated with the semigroup (\cite{BluGet68,chen-book})
\begin{align}
\label{semigR}
\mathbf{E}_x[f(X_t)] = \mathbf{E}_x[f({^*X}_t) M_t] 
\end{align}
where the multiplicative functional $M_t$ is associated to the Revuz measure given by the perturbation of the form $\mathcal{E}_c$. Thus, \eqref{semigR} is the solution to $\partial_t u = \Delta_R u$, $u_0=f \in D(\Delta_R)$. 

For $c=0$, the form $\mathcal{E}_c$ is associated with
\begin{align*}
\mathbf{E}_x[f(X_t)] = \mathbf{E}_x[f({^*X}_t)]
\end{align*}
solution to $\partial_t u = \Delta_N u$, $u_0=f \in D(A)$.

Now we prove that 
if $c_n \to \infty$, the sequence of forms $\mathcal{E}^n$ M-converges on $L^2(\Omega^{(\xi)})$ to the form 
$$
\mathcal{E}_\infty(u,u)=\begin{cases}  \displaystyle \int_{\Omega^{(\xi)}\setminus B} |\nabla u|^2\, dx dy & \hbox{\text for}\,u|_{\Omega^{(\xi)}\setminus B}\in
 H_0^1(\Omega^{(\xi)}\setminus B)\\
+\infty\quad &\hbox{\text otherwise.}\end{cases}$$

First we prove condition $ (a)$  of Definition \ref{d2.4}. Up to passing to a
subsequence, which we still denote by $v_n$,
we can suppose that 
\begin{equation}\label{e3.14} \,v_n |_{\Omega^{(\xi |n)}\setminus B}\in
 H^1(\Omega^{(\xi | n)}\setminus B),  \end{equation} and, for every $n,$
\begin{equation}\label{ee3.14} ||v_n||_{H^1(\Omega^{(\xi|n)}\setminus B)} \leqslant c^*, \end{equation} with $c^*$  independent of $n.$
First we extend $v_n$  by Jones extension operator   (Theorem 1 in \cite{Jones}) and after we restrict it to the domain $\Omega^{(\xi)}\setminus B$ : more precisely, we extend $v_n$ to a function $v^*_n=Ext_J\, v_n|_{\Omega^{(\xi)}\setminus B},$ such that \begin{equation}\label{eee3.14} ||v^*_n||_{H^1(\Omega^{(\xi)}\setminus B)}   \leqslant  C_J ||v_n||_{H^1(\Omega^{(\xi)}_n\setminus B)}  \leqslant C_J c^* . \end{equation}

We point out that the constant $C_J $ independent of $n$ (see Theorem 3.4 in  \cite{CAP})  that is  the norm of extension operator is independent of the (increasing) number of sides.

Then,  there exists $v^*$ such that   the sequence $v^*_n$ weakly converges to $v^*$ in $H^1(\Omega^{(\xi)}\setminus B):$ for the uniqueness of the limit in the weak topology, we obtain that  $v^*=u$ and, in particular, $u\in H^1(\Omega^{(\xi)}\setminus B).$
Since  the sequence $v^*_n$ weakly converges to $u$ in $H^1(\Omega^{(\xi)}\setminus B),$
we have that \begin{equation}\label{11} \underline{\lim}\ \int_{\Omega^{(\xi|n)}\setminus B} |\nabla v_n|^2\, dx dy\ge \int_{\Omega^{(\xi)}\setminus B} |\nabla u|^2\, dx dy.\end{equation}

 From the compact embedding of  $H^1(\Omega^{(\xi)})$ in $H^\alpha(\Omega^{(\xi)})$ ($\frac{1}{2}<\alpha<1$), we  have that
\begin{equation}\label{12s} ||v^*_n-u||_{H^\alpha(\Omega^{(\xi)}\setminus B)}\to 0\end{equation}	and by using Trace theorems (see \cite{JonWal} and  \cite{CAPVIV})
we obtain that 
\begin{equation}\label{bou}   \sigma^{(\xi | n)} \int_{\partial\Omega^{(\xi | n)}}
|v_n|^2 ds   \to k \int_{\partial\Omega^{(\xi)}}
| u|^2 d\mu^{(\xi)} \end{equation}  when $n\to\infty$ (see Theorem 2.1 in \cite{CAP}).
We stress the fact that the value of $\sigma^{(\xi | n)}$ play a crucial role in the previous limit.

Now, if  $c_n\to \infty,$ for any $k>0$ there exists $n_1$ such that, for all $n>n_1,$ $c_n\geq k.$
Then 
\begin{equation}\label{16++}  c_n \sigma^{(\xi | n)} \int_{\partial\Omega^{(\xi | n)}}
| v_n|^2 ds  \geq  k \sigma^{(\xi | n)} \int_{\partial\Omega^{(\xi | n)} }
| v_n|^2 ds    \to k \int_{\partial\Omega^{(\xi)}}
| u|^2 d\mu^{(\xi)} \end{equation}  when $n\to\infty.$
Dividing for $k$ and letting $k\to\infty$ we obtain that 
\begin{equation}\label{16+++}  \int_{\partial\Omega^{(\xi)}}
| u|^2 d\mu^{(\xi)}=0 \end{equation} and so $u=0$ on $\partial\Omega^{(\xi)}.$
By combining \eqref{11}, \eqref{bou}, \eqref{16+++}  we  have proved condition $ (a)$  of Definition \ref{d2.4}.

In order to prove condition (b) of Definition \ref{d2.4}, we can assume that $u\in H_0^1(\Omega^{(\xi)}\setminus B)$ without loss of generality: then, the choice of $v_n=u$ suffices to achieve the result. 
So we have proved the M-convergence of the forms $\mathcal{E}^n(\cdot,\cdot)$   on $L^2(\Omega^{(\xi)}\setminus B)$ to the form $\mathcal{E}_\infty$ when $c_n \to \infty$.

From the M-convergence of the forms $\mathcal{E}^n(\cdot,\cdot)$   on $L^2(\Omega^{(\xi)}\setminus B),$  by using the results in the recent paper  \cite{CapDovVIA}, we obtain the convergence of the time changed processes.

More precisely, from the $M$-convergence of the forms we have  the  strong convergence of semigroups. From Theorem 17.25 (Trotter, Sova, Kurtz, Mackevi\v{c}ius) in \cite{Kal} we have that strong convergence of semigroups (Feller semigroups) is equivalent to weak convergence of measures if $X^n_0 \to X_0$ in distribution. Then we obtain that $X^n \stackrel{d}{\to} X$ in $\mathbb{D}$. 

From Proposition \ref{teoK}, we have that
\begin{align*}
\forall\, \xi \in \Xi, \quad X^n \circ L =: X^{L, n} \to X^L := X \circ L \quad \textrm{on} \quad \Omega^{(\xi)}
\end{align*}
in distribution as $n\to \infty$ in $\mathbb{D}$.

From the pointwise convergence, we get that $\boldsymbol{\xi}-$a.s.
\begin{align*}
X^{L, n} \to X^L \quad \textrm{on} \quad \Omega^{(\boldsymbol{\xi})}
\end{align*}
in distribution as $n\to \infty$ in $\mathbb{D}$. 

\end{proof}

Let us consider now the process $X^L$ on $E$. We point out some peculiar aspects of $X^L$ and the corresponding lifetimes.
 
\begin{tm}
Let us consider the Cauchy problems
\begin{align}
\label{CP:ordinary}
\partial_t w = Aw, \quad w_0 = f \in D(A)
\end{align}
and
\begin{align}
\label{CP:fractional}
\mathfrak{D}^\Phi_t u  = Au, \quad u_0 = f \in D(A)
\end{align}
with $\Phi$ such that
\begin{align*}
\lim_{\lambda \to 0} \frac{\Phi(\lambda)}{\lambda}  \in (0, \infty).
\end{align*}
We have that, $\forall\, x\in E$:
\begin{itemize}
\item[-] if $\Phi^\prime (0) < 1$, then 
$$\int_0^\infty u(t,x)dt < \int_0^\infty w(t,x) dt,$$  
\item[-] if $\Phi^\prime (0) > 1$, then 
$$\int_0^\infty u(t,x)dt > \int_0^\infty w(t,x) dt,$$  
\item[-] if $\Phi^\prime (0) = 1$, then 
$$\int_0^\infty u(t,x)dt = \int_0^\infty w(t,x) dt.$$  
\end{itemize}
\end{tm}

\begin{proof}
The solution to \eqref{CP:ordinary} has the following probabilistic representation
\begin{align*}
w(t,x) = \mathbf{E}_x[f(X_t)] = \mathbf{E}_x[f({^*X}_t) M_t]
\end{align*}
where $M_t = \mathbf{1}_{(t < \zeta)}$ is the multiplicative functional written in terms of the lifetime $\zeta$ of the process $X$ on $E$. Then, we consider the part process $X$ of ${^*X}$ where ${^*X}_0=x \in E$. It is well known that $M_t$ characterizes uniquely the associated semigroup (\cite{BluGet68}), that is the solution $w$. We also have that   
\begin{align*}
\mathbf{E}_x\left[ \int_0^\zeta   f(X_s) ds\right] = \int_0^\infty w(t,x) dt =: \overline{w}(x)   
\end{align*}
is the solution to the elliptic problem on $E$
\begin{align*}
-A  \overline{w} = f.
\end{align*}
From Theorem \ref{time-frac-THM} we have that the time-changed  process $X^L$ can be considered in order to solve the problem \eqref{CP:fractional}, that is 
\begin{align*}
u(t, x) = \mathbf{E}_x[f(X_t^L)] = \mathbf{E}_x[f({^*X}^L_t) \mathbf{1}_{(t < \zeta^L)}] 
\end{align*}
where $\zeta^L$ is the lifetime of $X^L$. As before we introduce the 
\begin{align*}
\overline{u}(x) := \int_0^\infty u(t,x) dt = \mathbf{E}_x\left[ \int_0^{\zeta^L}   f(X_s^L) ds\right]
\end{align*}
which is the solution to the elliptic problem associated with the fractional Cauchy problem \eqref{CP:fractional}. We are able to obtain the key relation between $\overline{w}$ and $\overline{u}$ by taking into consideration the following plain calculations. First we recall \eqref{sol-time-frac-problem} where $P_sf(x)$ here is given by $w(s,x)$ with $w(s,x) \to f(x)$ as $s\to 0$. Moreover (see \cite{DelRus}),
\begin{align}
\label{ResL}
\int_0^\infty e^{-\lambda t} \mathbf{P}_0(L_t \in ds) dt = \frac{\Phi(\lambda)}{\lambda} e^{-s \Phi(\lambda)} ds.
\end{align}
We have that 
\begin{align*}
\overline{u}(x) = & \lim_{\lambda \to 0} \int_0^\infty e^{-\lambda t} u(t,x) dt \\
= & [\textrm{by} \; \eqref{sol-time-frac-problem} ]\\
= & \lim_{\lambda \to 0} \int_0^\infty e^{-\lambda t} \int_0^\infty w(s,x) \mathbf{P}_0(L_t \in ds)\, dt\\
= & [\textrm{by} \; \eqref{ResL} ]\\
= & \lim_{\lambda \to 0} \int_0^\infty w(s, x) \frac{\Phi(\lambda)}{\lambda} e^{-s \Phi(\lambda)} ds\\
=& \left( \lim_{\lambda \to 0} \frac{\Phi(\lambda)}{\lambda} \right) \int_0^\infty w(s,x) ds.
\end{align*}
That is
\begin{align*}
\overline{u}(x) = \left( \lim_{\lambda \to 0} \frac{\Phi(\lambda)}{\lambda} \right) \overline{w}(x)
\end{align*}
and this gives a connection between solutions of elliptic problems introduced above in the proof. Since $\Phi$ is a Bernstein function with $\Phi(0)=0$ we get the result. 
\end{proof}

The characterization given in the previous result admits a probabilistic interpretation in terms of mean lifetime of the base and time-changed processes. The problems \eqref{CP:ordinary} and \eqref{CP:fractional} with $f=\mathbf{1}_E$ are associated with $\overline{w}(x) = \mathbf{E}_x[\zeta]$ and $\overline{u}(x)=\mathbf{E}_x[\zeta^L]$ as described in the previous proof and the mean lifetime says how much the time change $L$ modifies the base process $X$. By following the definition given in \cite{DelRus} and the relation between $\overline{w}$ and $\overline{u}$ we say that $X$ is delayed or rushed on $E$ by $L$. An example is given by the tempered fractional derivative (\cite{Beghin, tempered}) associated with the symbol $\Phi(\lambda) = (\lambda + \eta)^\beta - \eta^\beta$ with $\eta>0$ and $\beta \in (0,1)$. We get that
\begin{align*}
\mathbf{E}_x[\zeta^L] =\beta \eta^{\beta -1} \mathbf{E}_x[\zeta]
\end{align*}
that is, if $\beta \eta^{\beta-1} < 1$ then the process $X$ is rushed by $L$,  whereas if $\beta \eta^{\beta-1} > 1$ then the process $X$ is delayed by $L$.

The previous discussion on either delayed or rushed processes holds according to specific regularity conditions on the boundary $\partial E$. We must have that $\sup_E w(x) < \infty$ which is the characterization of trap domains (written here for $X$ with generator $A$) given in \cite{BCM} for the Brownian motion. By applying the result in \cite{BCM} it follows that the following proposition holds true.

\begin{prop}
For $\xi \in \Xi$, the domains $\Omega^{(\xi| n)}$, $n \in \mathbb{N}$ are non trap for the Brownian motion.
\end{prop}

Since the previous statement holds pointwise for any contraction factor, we immediately obtain the following general statement.

\begin{prop}
For the $\Xi$-valued random vector $\boldsymbol{\xi}$ the domains $\Omega^{(\boldsymbol{\xi}| n)}$, $n \in \mathbb{N}$ are $\boldsymbol{\xi}-$a.s. non trap for the Brownian motion. 
\end{prop}

\bigskip 
 
\textbf{Grant.} The authors are members of GNAMPA (INdAM) and are partially supported by Grants Ateneo \lq\lq Sapienza" 2018.

\end{document}